\documentclass[EJP,preprint]{ejpecp}

\usepackage{mathtools,bm}
\usepackage[tt=false]{libertine}
\usepackage[varbb]{newpxmath}
\mathtoolsset{showonlyrefs=true}
\usepackage{booktabs,tabularx,array}
\usepackage{placeins}
\usepackage{microtype}
\usepackage{xcolor}
\microtypesetup{expansion=false}

\newcommand{\R}{\mathbb R}
\newcommand{\E}{\mathbb E}
\newcommand{\dd}{\,\mathrm d}
\newcommand{\Ent}{\operatorname{Ent}}
\newcommand{\TV}{\mathrm{TV}}
\newcommand{\cC}{\mathsf C_R}
\newcommand{\Rop}{\mathsf R_{h,L}}
\newcommand{\hRop}{\widehat{\mathsf R}_{h,L,-i}^{N}}
\newcommand{\1}{\boldsymbol 1}

\SHORTTITLE{Propagation of chaos for density-dependent diffusions}

\TITLE{Density-Dependent McKean--Vlasov Diffusions:\\
Subgaussian Occupancy Bounds and Polynomial Propagation of Chaos\support{The present research is supported by the Deutsche Forschungsgemeinschaft through the BE 3961/7-1 ``Statistische Inferenz f\"ur Teilchensysteme und McKean-Vlasov-SDEs mit singul\"aren Kernen''.}}

\AUTHORS{%
  Denis Belomestny\footnote{Duisburg-Essen University. \EMAIL{denis.belomestny@uni-due.de}}
  \and
  Ekaterina Morozova\footnote{Duisburg-Essen University. \EMAIL{ekaterina.morozova@uni-due.de}}
}

\AMSSUBJ{60K35; 60J60}

\SUBMITTED{July 20, 2026}
\ACCEPTED{}

\VOLUME{0}
\YEAR{2026}
\PAPERNUM{0}
\DOI{10.1214/YY-TN}

\ABSTRACT{%
We study the local density-dependent diffusion $\dd Y_t=-\Xi(p_t(Y_t))\nabla\Phi(Y_t)\dd t+\sqrt2\,\dd W_t$ and a clipped, randomly shifted histogram particle approximation on $\R^d$. The central difficulty is that the empirical density is evaluated at the particles' locations and re-enters their drift, while the confining force $\nabla\Phi$ may be unbounded. We provide a path-space entropy proof under two verifiable analytic conditions: a uniform pointwise Gaussian envelope for the true density $p_t$, and a Gaussian--polynomial bound for its spatial gradient $\nabla p_t$. The potential is allowed to have a gradient of at most linear growth.
The probabilistic input is a weighted exponential occupancy estimate under the independent product law. It is proved by Poissonizing the system at total intensity $N-1$, performing a one-cell leave-one-out estimate bounded via Poisson information, using Gaussian cell summability, and de-Poissonizing. For every fixed time horizon $T$, we obtain $\Ent(P_t^{N,k}\mid p_t^{\otimes k})\le C_T k(h^2(1+|\log h|)\allowbreak+(h^{-d}+\log N)/N)$. Consequently, selecting the optimally balanced bandwidth $h\asymp (N\log N)^{-1/(d+2)}$ yields a total variation error of $\|P_t^{N,k}-p_t^{\otimes k}\|_{\TV}\le C_T\sqrt{k}\,N^{-1/(d+2)}\allowbreak(\log N)^{d/(2(d+2))}$ for fixed $k$. This includes the usual Ornstein--Uhlenbeck density and the density-dependent OU model whenever the PDE estimates hold on the considered interval.
Furthermore, the histogram estimator offers a scalable approach for particle approximations. Using occupied-cell hashing, one algorithm step evaluates in expected $O(dLN)$ operations under standard constant-time hashing assumptions. For a fixed dimension and number of shifts, this requires expected $O(N)$ time, avoiding the $O(N^2)$ evaluation cost typical of standard kernel density estimators.
}

\begin{document}

\section{Introduction}

Classical McKean--Vlasov equations depend continuously on the law through integrals like $K*\mu$. Local density dependence is different: the map $\mu\mapsto (\dd\mu/\dd x)(x)$ is not continuous for weak or Wasserstein topologies. This destroys the direct Lipschitz coupling argument behind the standard propagation-of-chaos theory \cite{Sznitman1991,Meleard1996}. The problem is delicate for the SDE
\begin{equation}
 \dd Y_t=-\Xi(p_t(Y_t))\nabla\Phi(Y_t)\dd t+\sqrt2\,\dd W_t, \qquad \mathcal L(Y_t)(\dd x)=p_t(x)\dd x,
 \label{eq:limit-sde}
\end{equation}
because the local density is multiplied by a potentially unbounded confining force.
This model arises in Markov Chain Monte Carlo (MCMC) and generative modelling. By setting $m(r)=r\Xi(r)$ and $\mathsf h''(r)=1/(r\Xi(r))$, its Fokker--Planck equation has the generalized-mobility gradient-flow form $\partial_t p = \nabla\!\cdot(m(p)\nabla(\mathsf h'(p)+\Phi))$. Thus, a nonconstant $\Xi$ changes the mobility and the internal energy, rather than merely replacing the confining potential. For the reversible zero-flux stationary density characterized in \cite{BelomestnyMorozova}, one has $\pi(x)=g^{-1}(\mu-\Phi(x))$ with $g(r)=\int_1^r (s\Xi(s))^{-1}\dd s$, where $\mu$ normalizes the measure. A nonconstant mobility can reshape high- and low-density regions while maintaining a strongly confining potential. If one desires to sample from a prescribed target distribution $\pi_\star$, one must select $\Phi(x) = \mu - g(\pi_\star(x))$. Setting $\Phi = -\log \pi_\star$ is insufficient because a nonconstant mobility $\Xi$ generally changes the target distribution. While establishing finite-time propagation of chaos validates the particle approximation, algorithmic convergence (long-time mixing) is not addressed.
The standard particle approximation replaces $p_t(X_t^i)$ by a Kernel Density Estimate (KDE). Its constants involve $\|G_h\|_\infty\asymp h^{-d}$ or $\|\nabla G_h\|_\infty\asymp h^{-d-1}$. Inserting these into a direct Gronwall estimate creates an exponential bound with a negative power of $h$ in the exponent. The observation of this note is that the SDE only needs the drift error after applying the bounded function $\Xi$. Clipping a histogram density before applying $\Xi$ converts that error into a local occupancy statistic. Its exponential moment is then controlled by the Gaussian rarity of remote cells and by Poisson information, without any bandwidth-dependent Gronwall coefficient.
The histogram choice also resolves a main computational bottleneck in particle algorithms. Directly evaluating a KDE requires $O(N^2)$ distance computations per time step. With shifted histograms, one pass through the particles constructs a hash table of occupied-cell counts, taking expected $O(dLN)$ operations. 
Our contribution is fourfold. First, we provide a complete proof of the weighted product-law occupancy estimate on the whole space. Second, we insert it into a path-space Girsanov relative-entropy identity to obtain a propagation-of-chaos theorem avoiding discontinuous-drift PDE complications. Third, we isolate conditions checkable from PDE estimates: a Gaussian density envelope, Gaussian--polynomial control of the density gradient (allowing for small-time $t^{-1/2}$ singularities), and at-most-linear growth of $\nabla\Phi$. Fourth, we detail an expected $O(dLN)$ hashing implementation for algorithmic execution.

\section{Literature Overview}

Local density-dependent nonlinear diffusions and their particle approximations have a classical history. Oelschl\"ager \cite{Oelschlager1985,Oelschlager1987} established law-of-large-numbers and fluctuation results for moderately interacting diffusion processes, and M\'el\'eard and Roelly-Coppoletta \cite{MeleardRoelly1987} proved propagation of chaos for systems with moderate interaction. Most directly related to the present equation, Jourdain and M\'el\'eard \cite{JourdainMeleard1998} considered nonlinear SDEs whose drift and diffusion coefficients depend locally on the density of the time marginal. Under strong smoothness assumptions, they approximated the equation by smooth mollified moderately interacting systems and proved trajectorial propagation of chaos and fluctuation results. Relatedly, Chen et al.~\cite{ChenDausHolzingerJuengel2021} consider multi-species particle systems with nonlinear density-dependent diffusion coefficients and smooth shrinking interaction kernels. By passing through an intermediate nonlocal system, they derive the local cross-diffusion limit and prove a trajectory-level convergence estimate implying propagation of chaos. The present work differs in that the estimator is a discontinuous leave-one-out histogram, the density dependence enters the drift through the nonlinear transform $\Xi$, the confining force may grow linearly, and the conclusion is an explicit relative-entropy bound obtained through a weighted occupancy estimate.
Interacting particle systems are used in computational statistics for approximate sampling. Methods like Stein Variational Gradient Descent (SVGD) \cite{LiuWang2016} use deterministic pairwise-kernel interacting transport to distribute particles, typically requiring an $O(N^2)$ evaluation cost. Interacting Langevin diffusions such as the ensemble Kalman sampler \cite{GarbunoInigo2020} employ empirical covariance matrices for affine-invariant sampling. Theoretical guarantees for these systems often rely on bounded forces, compact domains, or highly regularized kernels. The framework introduced here allows unbounded forces, including strongly confining examples, whenever the required density estimates hold.
It is instructive to contrast the present approach with the modern moderate-interaction framework of Chen, Holzinger, and Huo \cite{ChenHolzingerHuo2025}. Their local limit comes from a smooth convolution kernel, and their proof uses a regularized $L^2$ coercive structure. Here, the estimated field is a discontinuous histogram evaluated at the tagged particle. The nonlinear bounded transform $\Xi$, the clipping operation, and the unbounded spatial weight prevent the use of $L^2$ coercivity and require the new occupancy estimate.
Other theories provide partial ingredients but do not directly cover a shrinking local-density estimator inside a nonlinear unbounded drift (see Table~\ref{tab:literature}). Local entropy hierarchies give marginal estimates for pairwise or smooth measure interactions \cite{Lacker2023,ArneseLacker2026}; higher-order $L^2$ expansions reveal connected correlations \cite{HessChildsRowan2025}; Fisher-information hierarchies handle smooth interactions \cite{GrassPoquetGuillin2025}. Statistical change-of-measure methods give Bernstein inequalities for previously controlled systems \cite{DellaMaestraHoffmann2022}, while conditional Hilbert-space bounds handle singular kinetic kernels \cite{HaoZhangZhao2026}. The required density bounds for the density-dependent OU model are supplied by Belomestny and Morozova \cite{BelomestnyMorozova}. 

\begin{table}[htbp]
\centering
\caption{Comparison with quantitative propagation-of-chaos methods.}
\label{tab:literature}
\renewcommand{\arraystretch}{1.16}
\footnotesize
\begin{tabularx}{\textwidth}{>{\raggedright\arraybackslash}p{0.19\textwidth} >{\raggedright\arraybackslash}p{0.22\textwidth} >{\raggedright\arraybackslash}X}
\toprule
Methodology & Main mechanism & Relation to the present model\\
\midrule
Sznitman \cite{Sznitman1991}; M\'el\'eard \cite{Meleard1996} & Synchronous coupling & Pointwise density evaluation is discontinuous in weak or Wasserstein topologies.\\
Oelschl\"ager \cite{Oelschlager1985,Oelschlager1987}; M\'el\'eard \& Roelly \cite{MeleardRoelly1987}; Jourdain \& M\'el\'eard \cite{JourdainMeleard1998}; Chen et al.~\cite{ChenDausHolzingerJuengel2021} & Smooth moderate interactions and local density dependence & Requires strong smoothness and regularized kernels; does not cover discontinuous histogram estimators.\\
Lacker \cite{Lacker2023}; Arnese \& Lacker \cite{ArneseLacker2026} & BBGKY local entropy hierarchies & Handles nonlinear measure dependence, but assumptions do not directly accommodate discontinuous shrinking estimators.\\
Hess--Childs \& Rowan \cite{HessChildsRowan2025} & $L^2$ hierarchy & Explains connected low-order terms but does not close the nonlinear histogram drift.\\
Chen et al.\ \cite{ChenHolzingerHuo2025} & Regularized $L^2$ relative entropy & Relies on a coercive convolution structure absent from our nonlinear drift.\\
Liu \& Wang \cite{LiuWang2016} & SVGD and related kernelized methods & Uses deterministic pairwise-kernel interacting transport creating an $O(N^2)$ evaluation cost.\\
Garbuno-Inigo et al.\ \cite{GarbunoInigo2020} & Ensemble Kalman sampler & Uses empirical covariance interactions rather than local density estimation.\\
Present note & Poissonized occupancy \& entropy transfer & Exploits leave-one-out histogram structure to absorb linear confinement growth.\\
\bottomrule
\end{tabularx}
\end{table}

\FloatBarrier

\section{Model, Estimator, and Analytic Assumptions}

Let $0<h\le1$, $v=h^d$, and fix deterministic shifts $U_1,\ldots,U_L\in[0,h)^d$. For each shift $U_\ell$, let $\{B_m^{U_\ell}:m\in\mathbb Z^d\}$ be a partition of $\R^d$ into half-open cubes of side $h$, such that cell membership is unambiguous on boundaries. The shifts may be sampled once and frozen; the following bounds hold conditionally and uniformly for every deterministic realization. 
We define the cell average $\Pi_{h,U_\ell}q(x)=\frac1v\int_{B_m^{U_\ell}}q(y)\dd y$ for $x\in B_m^{U_\ell}$. Using a standard clipping function $\cC(r)=r\wedge R$, the population field is defined as $\Rop q(x)=\frac1L\sum_{\ell=1}^L \cC(\Pi_{h,U_\ell}q(x))$.
For a discrete particle configuration $x=(x_1,\ldots,x_N)$, let $n_m^{U_\ell}(x)$ be the number of particles in $B_m^{U_\ell}$. If $x_i\in B_{m(i,\ell)}^{U_\ell}$, we define the leave-one-out empirical field as
\begin{equation}
 \hRop(x_i;x)=\frac1L\sum_{\ell=1}^L \cC\!\left( \frac{n_{m(i,\ell),\ell}(x)-1}{(N-1)v}\right).
 \label{eq:empirical-field}
\end{equation}
The interacting particle system is
\begin{equation}
 \dd X_t^{i,N} =-\Xi(\hRop(X_t^{i,N};X_t))\nabla\Phi(X_t^{i,N})\dd t +\sqrt2\,\dd W_t^i, \qquad i=1,\ldots,N.
 \label{eq:particle-system}
\end{equation}

\begin{assumption}
\label{ass:coefficients}
The function $\Xi:[0,\infty)\to(0,\infty)$ is bounded and globally Lipschitz, meaning $0<\kappa\le\Xi(r)\le K$ and $|\Xi(r)-\Xi(s)|\le L_\Xi|r-s|$. The potential $\Phi\in C^1(\R^d)$ has a gradient with at most linear growth: $|\nabla\Phi(x)|^2\le C_\Phi(1+|x|^2)$.
\end{assumption}

\begin{assumption}
\label{ass:density}
For a fixed time horizon $T<\infty$, \eqref{eq:limit-sde} is well posed, its density solves the PDE 
\begin{equation}
\partial_t p_t=\Delta p_t+ \nabla\!\cdot(p_t\Xi(p_t)\nabla\Phi),
\end{equation}
and there are $C_0,c_0,C_1,a_1>0$ and $m\ge0$ such that for $0 < t\le T$ and $x\in\R^d$:
\begin{equation}\label{eq:dens_Gaussbound}
 p_t(x)\le C_0e^{-c_0|x|^2}
 \end{equation}
 and
 \begin{equation}\label{eq:grad_Gaussbound}
 |\nabla p_t(x)|\le C_1\max\bigl\{1,t^{-1/2}\bigr\}(1+|x|)^m e^{-a_1|x|^2}.
\end{equation}
The algorithmic clipping level satisfies $R>C_0$.
\end{assumption}
\begin{remark}\label{rem:p0}
	The uniformity of Gaussian bounds in Assumption~\ref{ass:density} as \(t\to 0\) implies the same bound for the initial law. Indeed, let \(\mu_t(dx)=p_t(x)\,dx\) and \(\nu(dx)=C_0e^{-c_0|x|^2}\,dx\). Continuity of the paths of \(Y\) implies that \(\mu_t\to\mu_0\) weakly as \(t\to 0\). Since \(\mu_t\le\nu\) for every \(t>0\), the Portmanteau theorem gives, for every open set \(G\subset\mathbb R^d\),
\[
\mu_0(G)\le\liminf_{t\to 0}\mu_t(G)\le\nu(G).
\]
By outer regularity, \(\mu_0\le\nu\) on all Borel sets. Consequently, \(\mu_0\) is absolutely continuous and its density satisfies
\[
p_0(x)\le C_0e^{-c_0|x|^2}
\qquad\text{for almost every }x\in\mathbb R^d.
\]
In particular, \(p_0\) has finite moments of every polynomial order. 
\end{remark}

\section{Main Propagation-of-Chaos Theorem}

Let $P_t^N$ denote the time-marginal law of \eqref{eq:particle-system}, initialized from i.i.d.\ draws $P_0^N=p_0^{\otimes N}$. Let $P_t^{N,k}$ be its $k$-particle marginal. By evaluating the relative entropy on the path space using Girsanov's theorem, we bypass the need for density regularity of the discontinuous empirical SDE.

\begin{lemma}
\label{lem:entropy-production}
Let $\mathbb{Q}^N$ be the law of $N$ independent copies of the limit diffusion \eqref{eq:limit-sde}, so that $Q_t^N = p_t^{\otimes N}$. Let $\mathbb{P}^N$ be the law of the interacting particle system \eqref{eq:particle-system} on the canonical path space $C([0,T]; (\R^d)^N)$. Then
\begin{equation}
    \sup_{i\le N} \E_{\mathbb{P}^N} \sup_{s\le T} |X_s^{i,N}|^2 < \infty,
\end{equation}
the particle system admits a unique weak solution, and the marginal relative entropy satisfies
\begin{equation}
    \Ent(P_t^N \mid Q_t^N) \le \frac14 \int_0^t \E_{P_s^N} |b_s^N - \bar b_s|^2 \dd s,
\end{equation}
where $b^N$ and $\bar b$ are the interacting and reference drifts, respectively.
\end{lemma}

\begin{proof}
First note that \(b^N\) is Borel measurable by Borel measurability of the histogram counts. Since by Assumption~\ref{ass:coefficients} \(\Xi\) is bounded and \(|\nabla\Phi(x)|\) has at most linear growth, 
\begin{equation}\label{eq:drift-bounds}
|b_i^N(x)|^2\le K^2C_\Phi(1+|x_i|^2)
\quad\text{and}{\quad}
|b^N(x)|^2\le C(N+|x|^2).
\end{equation}
Moreover, by Remark~\ref{rem:p0} the initial density \(p_0\) has a finite second moment. 
For \(M\ge1\), choose a bounded cutoff \(\chi_M:(\R^d)^N\to[0,1]\) equal to one on \(\{\max_i|x_i|\le M\}\) and zero on \(\{\max_i|x_i|\ge M+1\}\), and set \(b^{N,M}=\chi_Mb^N\). This is a bounded Borel drift. Let \(\mathbb W^N\) be the law under which
\[
X_t^N=X_0^N+\sqrt2\,W_t,\qquad X_0^N\sim p_0^{\otimes N},
\]
with \(X_t^N=(X_t^{1,N},\dots,X_t^{N,N})\) and \(W\) an \((\R^d)^N\)-valued Brownian motion independent of \(X_0^N\). Since \(b^{N,M}\) is bounded, Novikov's condition holds and
\[
Z_T^M
=
\exp\left(
\frac1{\sqrt2}\int_0^T b^{N,M}(X_s^N)\cdot dW_s
-\frac14\int_0^T|b^{N,M}(X_s^N)|^2\,ds
\right)
\]
is a true martingale. Under the probability measure \(d\mathbb P^{N,M}=Z_T^M\,d\mathbb W^N\), by the Girsanov's theorem,
\[
W_t^M=W_t-\frac1{\sqrt2}\int_0^t b^{N,M}(X_s^N)\,ds
\]
is Brownian motion and
\begin{equation}\label{eq:part-trunc}
X_t^N=X_0^N+\int_0^t b^{N,M}(X_s^N)\,ds+\sqrt2\,W_t^M.
\end{equation}
Thus the truncated equation has a weak solution. Conversely, let \((X^N,W,\mathbb P)\) be any weak solution of the truncated equation. Since \(b^{N,M}\) is bounded, the reverse Girsanov exponential is a true martingale and transforms \(X^N\) into the driftless process \(X_t^N=X_0^N+\sqrt2 B_t\) under a measure \(\mathbb Q\), whose law is \(\mathbb W^N\). The inverse Radon-Nikodym density is the forward Girsanov exponential evaluated along this canonical driftless process and is therefore uniquely determined by \(\mathbb W^N\). Hence the law of \(X^N\) under \(\mathbb P\) is unique.

The componentwise bound in~\eqref{eq:drift-bounds} applied to~\eqref{eq:part-trunc}, alongside the Cauchy-Schwarz and Doob's inequalities, give, for \(t\le T\)
\begin{eqnarray*}
\mathbb E_{\mathbb P^{N,M}}\sup_{u\le t}|X_u^{i,N}|^2
&\leq&
3\mathbb E|X_0^{i,N}|^2
+3K^2 C_\Phi t\int_0^t
\left(
1+\mathbb E_{\mathbb P^{N,M}}\sup_{r\le s}|X_r^{i,N}|^2
\right)ds\\
&&\hspace{6.5cm}+6\mathbb E_{\mathbb P^{N,M}}\sup_{u\le t}|W_u^{M,i}|^2\\
&\leq&
3\mathbb E|X_0^{i,N}|^2
+3K^2 C_\Phi T^2 + 24dT 
+3K^2 C_\Phi T\int_0^t 
\mathbb E_{\mathbb P^{N,M}}\sup_{r\le s}|X_r^{i,N}|^2\,ds.
\end{eqnarray*}
Gronwall's inequality therefore yields
\begin{equation}\label{eq:Gronwall-supbound}
\sup_{M\ge1}\sup_{1\le i\le N}
\mathbb E_{\mathbb P^{N,M}}
\sup_{s\le T}|X_s^{i,N}|^2
\le C_T,
\end{equation}
where 
\(
C_T:=(3\mathbb E|X_0^i|^2
+3K^2 C_\Phi T^2 + 24dT)\exp(3K^2 C_\Phi T^2)
\) 
is independent of \(M,N,h,L\), and the grid shifts. It then holds for
\[
\tau_M=\inf\{t\ge0:\max_{i\le N}|X_t^{i,N}|\ge M\},
\]
that
\begin{equation}\label{eq:sttime-bound}
\mathbb P^{N,M}(\tau_M\le T)
\le
\frac1{M^2}\sum_{i=1}^N
\mathbb E_{\mathbb P^{N,M}}\sup_{s\le T}|X_s^{i,N}|^2
\le \frac{NC_T}{M^2}.
\end{equation}
For \(M'>M\), the drifts \(b^{N,M'}\) and \(b^{N,M}\)
both coincide with \(b^N\) on
\(\{\max_{i\le N}|x_i|<M\}\). Hence, by weak uniqueness of the
truncated equations, the corresponding solutions stopped at
\(\tau_M\) have the same law. The truncated solutions may
therefore be pasted together to define a weak solution of the
original equation up to
\(
\tau_\infty:=\lim_{M\to\infty}\tau_M.
\)
Denoting its law by \(\mathbb P^N\), since the law under
\(\mathbb P^N\), stopped at \(\tau_M\), is the same as the
corresponding stopped law under \(\mathbb P^{N,M}\),
estimate~\eqref{eq:sttime-bound} implies
\(
\mathbb P^N(\tau_\infty\le T)
=
\lim_{M\to\infty}\mathbb P^N(\tau_M\le T)
=0,
\)
yielding that the solution is non-explosive. Moreover, any other weak
solution has, after stopping at \(\tau_M\), the law of the
\(M\)-truncated solution. Letting \(M\to\infty\) therefore proves
uniqueness in law for the original equation. Furthermore, for every \(M\),
\[
\mathbb E_{\mathbb P^N}
\left[
\sup_{s\le T\wedge\tau_M}|X_s^{i,N}|^2
\right]
=
\mathbb E_{\mathbb P^{N,M}}
\left[
\sup_{s\le T\wedge\tau_M}|X_s^{i,N}|^2
\right]\le
\mathbb E_{\mathbb P^{N,M}}
\left[
\sup_{s\le T}|X_s^{i,N}|^2
\right]
\le C_T.
\]
Since \(T\wedge\tau_M\to T\) a.s., it also holds, \(\mathbb{P}^N\)-a.s., \(\sup_{s\le T\wedge\tau_M}|X_s^{i,N}|^2
\to
\sup_{s\le T}|X_s^{i,N}|^2
\), and by the Fatou's lemma, 
\[
\mathbb E_{\mathbb P^N}
\left[
\sup_{s\le T}|X_s^{i,N}|^2
\right]
=
\mathbb E_{\mathbb P^N}
\left[
\liminf_{M\to\infty}
\sup_{s\le T\wedge\tau_M}|X_s^{i,N}|^2
\right]
\le
\liminf_{M\to\infty}
\mathbb E_{\mathbb P^N}
\left[
\sup_{s\le T\wedge\tau_M}|X_s^{i,N}|^2
\right]
\le C_T.
\]
Since the bound above holds for every \(i\) with the same constant,
\[
\sup_{1\le i\le N}
\mathbb E_{\mathbb P^N}\sup_{s\le T}|X_s^{i,N}|^2
\le C_T.
\]

Now consider \(N\) independent copies of the limit process. By Assumption~\ref{ass:coefficients}, the coefficients \(\bar b_t(x)\) and \(\Delta b_t:=b^N-\bar b_t\) satisfy the same bound as in~\eqref{eq:drift-bounds}, and the same moment argument as above applies under \(\mathbb Q^N\). Fix \(t\le T\) and consider the canonical stopping time
\[
\tau_M(\omega)=\inf\{s\ge0:\max_{i\le N}|\omega_s^i|\ge M\}.
\]
Since \(\Delta b_s(X_s^N)1_{\{s\le\tau_M\}}\) is bounded, the exponential
\[
Z_{t,M}
=
\exp\left(
\frac1{\sqrt2}\int_0^{t\wedge\tau_M}
\Delta b_s(X_s^N)\cdot dW_s^{\mathbb Q}
-\frac14\int_0^{t\wedge\tau_M}
|\Delta b_s(X_s^N)|^2\,ds
\right),
\]
with \(W^{\mathbb Q}\) a \(\mathbb Q^N\)-Brownian motion,
is a true \(\mathbb Q^N\)-martingale. The corresponding change of measure changes the drift from \(\bar b\) to \(b^N\) before \(\tau_M\). By stopped weak uniqueness, the resulting law is the stopped interacting law. Denoting the stopped path laws on \([0,t]\) by \(\mathbb P_{[0,t]}^{N,M}\) and \(\mathbb Q_{[0,t]}^{N,M}\), we therefore have
\(
d\mathbb P_{[0,t]}^{N,M}/
     d\mathbb Q_{[0,t]}^{N,M}
=Z_{t,M}.
\)
Under \(\mathbb P_{[0,t]}^{N,M}\),
\[
W_s^{\mathbb P}
=
W_s^{\mathbb Q}
-\frac1{\sqrt2}\int_0^{s\wedge\tau_M}\Delta b_r(X_r^N)\,dr
\]
is Brownian motion. Substitution into \(\log Z_{t,M}\), followed by taking expectations, gives
\[
\operatorname{Ent}\left(
\mathbb P_{[0,t]}^{N,M}\mid
\mathbb Q_{[0,t]}^{N,M}
\right)
=
\frac14\,
\mathbb E_{\mathbb P^{N,M}}
\int_0^{t\wedge\tau_M}
|\Delta b_s(X_s^N)|^2\,ds,
\]
because the stopped stochastic integral has zero expectation. The factor \(1/4\) results from the diffusion matrix \(\sqrt2 I\).

Finally, let \(P_t^{N,M}\) and \(Q_t^{N,M}\) denote the corresponding stopped time-\(t\) marginals. By the data-processing inequality,\[
\operatorname{Ent}(P_t^{N,M}\mid Q_t^{N,M})
\le
\frac14\,
\mathbb E_{\mathbb P^N}
\int_0^{t\wedge\tau_M}
|\Delta b_s(X_s^N)|^2\,ds,
\]
where the expectation may be taken under \(\mathbb P^N\), since the stopped laws coincide up to \(\tau_M\). Since again by non-explosion
\(
P_t^{N,M}\to P_t^N
\) 
and
\(
Q_t^{N,M}\to Q_t^N
\)
weakly,
joint lower semicontinuity of relative entropy and monotone convergence now give
\begin{multline*}
\operatorname{Ent}(P_t^N\mid Q_t^N)
\le
\liminf_{M\to\infty}
\operatorname{Ent}(P_t^{N,M}\mid Q_t^{N,M})\\
\le
\frac14\,
\mathbb E_{\mathbb P^N}
\int_0^t
|b^N(X_s^N)-\bar b_s(X_s^N)|^2\,ds
=
\frac14\int_0^t
\mathbb E_{P_s^N}|b^N-\bar b_s|^2\,ds.
\end{multline*}
This proves the result.
\end{proof}

\begin{theorem}
\label{thm:main-poc}
Under Assumptions~\ref{ass:coefficients} and \ref{ass:density}, there exists a constant $C_T<\infty$, independent of $N\ge2$, $0<h\le1$, $L$, and the shift locations, such that for $0\le t\le T$:
\begin{equation}
 \Ent\!\left(P_t^N\mid p_t^{\otimes N}\right) \le C_T\bigl(Nh^2(1+|\log h|)+h^{-d}+\log N\bigr).
\label{eq:full-entropy-main}
\end{equation}
Consequently, for any $1\le k\le N$, the marginal entropy and total variation distance satisfy:
\begin{align}
 \Ent\!\left(P_t^{N,k}\mid p_t^{\otimes k}\right) &\le C_T k\left(h^2(1+|\log h|)+\frac{h^{-d}+\log N}{N}\right), \label{eq:local-entropy-main}\\
 \|P_t^{N,k}-p_t^{\otimes k}\|_{\TV} &\le C_T\sqrt{k} \left(h\sqrt{1+|\log h|}+\sqrt{\frac{h^{-d}+\log N}{N}}\right). \label{eq:tv-main}
\end{align}
Furthermore, 
%if $b_i^N$ is the empirical drift in \eqref{eq:particle-system} and $\bar b_t(x)=-\Xi(p_t(x))\nabla\Phi(x)$, 
the time-integrated expected drift error obeys
\begin{equation}
 \int_0^T\E_{P_t^N}|b_i^N(X_t)-\bar b_t(X_t^{i,N})|^2\dd t \le C_T\left(h^2(1+|\log h|)+\frac{h^{-d}+\log N}{N}\right)
 \label{eq:drift-chaos-main}
\end{equation}
for any \(i=1,\dots,N\). The constant $C_T$ depends on the parameters $T$, $d$, $C_\Phi$, $\kappa$, $K$, $L_\Xi$, \(R\), $C_0$, $c_0$, $C_1$, $a_1$, and $m$. For fixed analytic constants, \(C_T\) tracks the application of Gronwall's inequality and may grow exponentially like $e^{C T}$ or $e^{C T}$ times a polynomial. 
\end{theorem}

\begin{corollary}
\label{cor:optimized}
Setting the bandwidth to $h_N\asymp (N\log N)^{-1/(d+2)}$ balances the polynomial and logarithmic error terms, yielding 
\begin{align}
 \sup_{t\le T} \Ent(P_t^{N,k}\mid p_t^{\otimes k}) &\le C_TkN^{-2/(d+2)}(\log N)^{d/(d+2)}, \\
 \sup_{t\le T}\|P_t^{N,k}-p_t^{\otimes k}\|_{\TV} &\le C_T\sqrt{k}\,N^{-1/(d+2)}(\log N)^{d/(2(d+2))}.
\end{align}
\end{corollary}

\begin{corollary}
\label{cor:growing-marginals}
The result above allows growing marginal sizes $1\le k_N\le N$. Provided 
\[
k_N N^{-2/(d+2)}(\log N)^{d/(d+2)} \to 0,
\]
 the relative entropy vanishes asymptotically, extending the propagation of chaos to growing ensembles.
\end{corollary}

\paragraph{Proof Strategy.}
If $(N-1)h^dr_m\ll1$, isolated particles register zero leave-one-out error. Collisions are rare and handled by a one-cell Poisson bound. For dense cells, the squared density fluctuation is $r_m/a$. Multiplication by the particle count scales the cell error cost to $r_m^2$, making it globally summable. In remote cells, the Gaussian envelope restricts spatial weight growth to $1+|x|^2\lesssim\log(1/r_m)$. The $k \log k$ rate function of the Poisson distribution overpowers and absorbs this logarithmic spatial weight. Using disjoint histograms is crucial: after Poissonization, cell counts become independent. The remaining non-summable terms are the Gaussian lattice sum (yielding $h^{-d}$) and the cost of conditioning a Poisson sum to equal $N$ (yielding $\log N$). Because we evaluate these bounds entirely under the independent product reference measure, neither term enters a Gronwall argument. Finally, the $t^{-1/2}$ singularity in the true density gradient resolves into a time-integrated logarithmic loss $1+|\log h|$. The proofs rely on the bias lemma deferred to Appendix~\ref{app:proofs}.

\begin{lemma}
\label{lem:weighted-bias}
Under Assumptions~\ref{ass:coefficients} and \ref{ass:density}, the cell averages inherit Gaussian decay $\Pi_{h,U}p_t(x) \le C e^{-c'|x|^2}$, and the deterministic bias 
\begin{equation}
\beta_h(t) := \sup_{x\in\R^d} |\nabla\Phi(x)|^2 |\Xi(\Rop p_t(x))-\Xi(p_t(x))|^2
\end{equation}
satisfies $\beta_h(t) \le C_T \min\{1,h^2/t\}$ uniformly across shifts.
\end{lemma}

\section{Verification for Ornstein--Uhlenbeck Models}
Setting \(\Xi\equiv\kappa_0\) with some \(\kappa_0>0\) and \(\Phi(x)=|x|^2\) in~\eqref{eq:limit-sde}, one obtains the classical Ornstein-Uhlenbeck process. Since its invariant density is exactly Gaussian, this model satisfies our assumptions under stationary initialisation; this can also be extended to a non-stationary initial condition, provided that the suitable Gaussian bounds are imposed on the initial density. For the case of non-constant \(\Xi\), our results become applicable under assumptions of~\cite{BelomestnyMorozova}.

\begin{corollary}
Suppose that \(\Phi(x)=1+|x|^2\), the drift \(\Xi\) satisfies Assumption~\ref{ass:coefficients} with \(K/\kappa\leq 2\), and the initial density \(p_0\in L^\infty(\R^d)\cap L^1(\R^d)\) is such that \(p_0(x)\leq C_p e^{-b_0|x|^2}\) with \(C_p>0\) and \(b_0>\Xi(0)\). Then the model~\eqref{eq:limit-sde} satisfies Assumption~\ref{ass:density} with \(c_0=\Xi(0)\), \(m=1\) and \(a_1=\Xi(0)/4\); see~\cite{BelomestnyMorozova}. Consequently, the histogram particle approximation initialized from \(p_0^{\otimes N}\) satisfies the entropy, total-variation, and time-integrated drift estimates~\eqref{eq:full-entropy-main}--\eqref{eq:drift-chaos-main}, as well as the optimized rates of Corollaries~\ref{cor:optimized} and~\ref{cor:growing-marginals}.
\end{corollary}

Since, as shown in~\cite{BelomestnyMorozova}, the unique stationary density satisfies \(\nabla\pi(x)=-\pi(x)\Xi(\pi(x))\nabla\Phi(x)\) and hence Gaussian upper bounds, our assumptions are also fulfilled for the system initialised at stationarity, ensuring that the propagation-of-chaos bounds hold.

\section{Euler Particle Scheme and Scalability}
\label{sec:euler}

A KDE or pairwise evaluation algorithm requires evaluating $O(N^2)$ distances at every time step. By contrast, our shifted histogram grid formulation yields an algorithmic step requiring only local counts. For a fixed step $\Delta t>0$, clipping level $R$, and $L$ shifts, one step from $t_n$ to $t_{n+1}$ is:
\begin{enumerate}
 \item Hash the $N$ particles into every shifted grid to compute the occupied cell counts $(n_m^{U_\ell})^n$.
 \item For each particle $i$, look up its corresponding cell counts and evaluate the scalar leave-one-out feature $\widehat r_{n,-i}^{h,L} =\frac1L\sum_{\ell=1}^L \cC( (n_{m(i,\ell),\ell}^n-1)/((N-1)h^d) )$.
 \item Execute the simultaneous Euler--Maruyama update: $X_{n+1}^i=X_n^i -\Delta t\,\Xi(\widehat r_{n,-i}^{h,L})\nabla\Phi(X_n^i) +\sqrt{2\Delta t}\,\xi_n^i$, with $\xi_n^i\sim\mathcal N(0,I_d)$.
\end{enumerate}
Computing a cell index requires \(O(d)\) operations, and each particle performs one hash-table update and one lookup per grid.
Under the standard expected \(O(1)\) complexity of hash-table
access, the density-feature stage therefore requires expected
\(O(dLN)\) operations per time step. Here the expectation refers
to the hash-table implementation, rather than to the probability
law of the particle positions; in particular, the complexity
bound applies under the interacting particle law \(P_t^N\).
For fixed \(d\) and \(L\), this gives expected \(O(N)\) complexity,
avoiding the \(O(N^2)\) distance evaluations required by a direct
KDE implementation.
%Computing cell indices and accumulating counts requires expected $O(1)$ operations per particle per grid, assuming expected constant-time hash access under standard hashing models. The expectation is with respect to the particle positions, which are random under the product law. The entire density-feature stage therefore costs expected $O(dLN)$ time. For fixed $d$ and $L$, this is expected $O(N)$ time, successfully bypassing $O(N^2)$ distance matrices.

\begin{remark}
\label{rem:euler-error}
The current paper establishes the continuous-time mean-field limit. The discretization error analysis for the unbounded-force case is left for future work; for the bounded-force case, the work of Jourdain and Menozzi \cite{JourdainMenozzi2024} provides a route. Assuming the bounded-force condition $\|\nabla\Phi\|_\infty\le M_\Phi<\infty$ (which models an at-most-linear potential globally on $\R^d$, not a compact domain), the full particle drift is a bounded measurable function. Under these conditions, randomized-time Euler schemes yield a pointwise transition density error bounded by $C\Delta^{1/2}(1+\log(T/\Delta))$ multiplied by a normalized Gaussian. Integrating this spatially yields an $O(\Delta^{1/2}(1+\log(T/\Delta)))$ total-variation bound for fixed $N$. However, this analytical approach relies on time-randomization, claims no polynomial dependence of the constant on $N$, and does not cover the unbounded Ornstein--Uhlenbeck force required by strongly confining sampling regimes.
\end{remark}

\section{Conclusion}

The clipped shifted-histogram formulation resolves the missing product-law bounds for interacting diffusions. By combining independent Poissonization under the uncoupled reference measure with integral entropy transfer, we establish a framework for local density estimation in McKean--Vlasov SDEs. Under Gaussian-envelope PDE conditions, the interacting particle error is polynomial, achieving balanced marginal entropy rates of $O(N^{-2/(d+2)}(\log N)^{d/(d+2)})$ and total-variation rates of $O(N^{-1/(d+2)}(\log N)^{d/[2(d+2)]})$. While this establishes finite-time propagation of chaos, the discretization error analysis for the unbounded-force case, algorithmic convergence to a prescribed target distribution, long-time mixing or ergodicity of the particle system, and uniform-in-time particle accuracy remain open.

\appendix

\section{Weighted Exponential Occupancy under the Product Law}
\label{sec:weighted-occupancy}

At a fixed time \(s\in[0,T]\), let $Y_1,\ldots,Y_N$ be i.i.d.\ draws from a target density \(p_s\), and denote \(Q_s^N=p_s^{\otimes N}\). For a cell $B_m^U$, define its local probability mass $P_m^U=\int_{B_m^U}p(x)\dd x$, nominal density $r_m^U=P_m^U/v$, and spatial weight $w_m^U=1+\sup_{x\in B_m^U}|x|^2$. Define the exact one-grid energy pathwise as 
\begin{equation}
    S_N = \sum_m n_m w_m \left| \cC\left(\frac{(n_m-1)_+}{(N-1)v}\right) - \cC(r_m) \right|^2.
\end{equation}

\begin{theorem}
\label{thm:weighted-mgf}
If Assumption~\ref{ass:coefficients} holds with \(L_\Xi>0\) and the density satisfies~\eqref{eq:dens_Gaussbound}, there exist $\alpha_*,C_*\in(0,\infty)$, depending only on $d,C_0,c_0,R,L_\Xi$, such that 
\begin{equation}
\log\E_{Q_s^N}\exp\left\{ \alpha_*\sum_{i=1}^N(1+|Y_i|^2) \left|\Xi(\hRop(Y_i;Y))-\Xi(\Rop p_s(Y_i))\right|^2 \right\} \le C_*\bigl(h^{-d}+\log N\bigr)
\end{equation}
uniformly in \(s\in[0,T]\) (in case \(L_\Xi=0\) the conclusion is trivial). 
The constants are uniform in $N\ge2$, bandwidth $0<h\le1$, shifts $U$, and across any family of densities satisfying identical envelope bounds.
\end{theorem}

To prove the theorem above, we need a few auxiliary results.
\begin{lemma}
\label{lem:poisson-information}
Let \(I_\lambda(k)=k\log(k/\lambda)-k+\lambda\) be the Poisson rate function.
If $K_\lambda\sim\operatorname{Poi}(\lambda)$, then for $0<\theta<1$, $\sup_{\lambda>0}\E e^{\theta I_\lambda(K_\lambda)}<\infty$. Furthermore, if $0<\theta<\theta'<1$, then $\sup_{\lambda>0} \E\bigl[(1+I_\lambda(K_\lambda)) e^{\theta I_\lambda(K_\lambda)}\bigr]<\infty$.
\end{lemma}
\begin{proof}
For $k=0$, $I_\lambda(0) = \lambda$, and $\mathbb{P}(K_\lambda=0)e^{\theta I_\lambda(0)} = e^{-(1-\theta)\lambda} \le 1$. By Stirling's bound for $k\ge1$, the Poisson probability satisfies $\mathbb P(K_\lambda=k)e^{\theta I_\lambda(k)} \le C k^{-1/2} e^{-(1-\theta)I_\lambda(k)}$. For $0 < \lambda \le 2$, we use the uniform inequality $I_\lambda(k) \ge k\log(k/2) - k$. The sequence $k^{-1/2} e^{-(1-\theta)(k\log(k/2) - k)}$ decays rapidly, and the finitely many small $k$ can be bounded individually. Thus the sum converges uniformly for $\lambda \in (0,2]$. 
For $\lambda > 2$, observing the convexity of $u\log u-u+1$, we partition the integer summation into three regimes. For $k<\lambda/2$, $I_\lambda(k)\ge c\lambda$; the exponential factor $e^{-c(1-\theta)\lambda}$ absorbs the $O(\sqrt{\lambda})$ summation over $k^{-1/2}$. For $\lambda/2\le k\le2\lambda$, $I_\lambda(k)\ge c(k-\lambda)^2/\lambda$; the discrete Gaussian summation has size $O(\sqrt{\lambda})$ which offsets the $k^{-1/2} \asymp \lambda^{-1/2}$ penalty. For $k>2\lambda$, $I_\lambda(k)\ge ck\log(k/\lambda)$. Because $k/\lambda > 2$, we have $I_\lambda(k) \ge c k \log 2$. The sum $\sum_{k>2\lambda} k^{-1/2} e^{-(1-\theta)ck\log 2}$ is bounded by a convergent geometric series independent of $\lambda$. Applying these bounds guarantees the expectation is bounded uniformly across all $\lambda>0$. Applying the elementary inequality $(1+x)e^{\theta x}\le C_{\theta,\theta'}e^{\theta'x}$ yields the second claim.
\end{proof}

\begin{lemma}
\label{lem:deterministic-cell}
Fix
\(0<A<R<\infty\) and \(B<\infty\).
Let $a>0$, $0<r\le A$, and define $\log_+(x) = \max\{0, \log x\}$. Assume $w\le B\bigl(1+\log_+(A/r)\bigr)$. Set $\lambda=ar$ and define $x_k=(k-1)_+/a$, $Z(k)=kw|\cC(x_k)-\cC(r)|^2$. There exist constants $C<\infty$, $r_0\in(0,1)$ and $c>0$ ensuring:
\begin{itemize}
 \item $Z(k)\le C(I_\lambda(k)+1)$ for all $k\in\{0,1,2\dots\}$.
 \item If $r\le r_0$ and $x_k\le\sqrt r$, then $Z(k)\le Cwr(I_\lambda(k)+1)$.
 \item If $r\le r_0$ and $x_k>\sqrt r$, then $I_\lambda(k)\ge c\bigl(1+\log(A/r)\bigr)$ and $Z(k)\le CI_\lambda(k)$.
\end{itemize}
\end{lemma}
\begin{proof}
For $k\le 1$, $x_k=0$ yielding $Z(1)\le wr^2$, which is bounded by a constant $C$ because $r^2(1+\log_+(A/r))$ remains bounded on $(0,A]$. For $k\ge 2$, if $k\le2\lambda$, then $\lambda\ge 1$. Since \(\cC\) is 1-Lipschitz and \(\lambda=ar\), the squared error satisfies $k|x_k-r|^2 \le 2kr^2\lambda^{-2}((k-\lambda)^2+1) \le Cr^2(I_\lambda(k)+1)$. Multiplication by $w$ yields the global bound and near-mean bounds since $wr^2 \le C$ and $r^2 \le Ar$. If $k>2\lambda$, let $u=k/\lambda>2$. Because the clipping function $\cC$ is bounded by $R$ and is 1-Lipschitz, the error ratio is bounded by $Z(k)/I_\lambda(k) \le Cw \min\{r^2u^2,R^2\} / \log u$. If $x_k \le \sqrt{r}$, then $ur = k/a \le 2\sqrt{r}$, making the un-logged numerator bounded by $Cwr$, confirming the near estimate. If $x_k > \sqrt{r}$, the ratio $u = k/\lambda \ge x_k/r > r^{-1/2}$. Shrinking $r_0$ ensures this is larger than 2, yielding $I_\lambda(k) \ge c k \log u$. Since $u > r^{-1/2}$, we have $\log u \ge \tfrac12\log(1/r)$, and thus $1+\log_+(A/r) \le C\log u$. Consequently,
\begin{equation}
    \frac{Z(k)}{I_\lambda(k)} \le C\frac{w\min\{r^2u^2,R^2\}}{\log u} \le C,
\end{equation}
and $I_\lambda(k) \ge c k \log u \ge c\bigl(1+\log(A/r)\bigr)$. This explicitly establishes the far estimate.
\end{proof}

\begin{lemma}
\label{lem:one-cell}
Under the assumptions of Lemma~\ref{lem:deterministic-cell}, there exist $0<\alpha_0,C<\infty$ and $\rho>0$ such that for $K\sim\operatorname{Poi}(ar)$ it holds $\log\E\exp\{ \alpha_0Kw |\cC((K-1)_+/a)-\cC(r)|^2 \} \le Cr^\rho$.
\end{lemma}
\begin{proof}
Let $Z=Z(K)$ and fix $0<\theta<\theta'<1$. If $r \ge r_0$, setting a sufficiently small $\alpha_0$ combined with Lemmas~\ref{lem:poisson-information} and~\ref{lem:deterministic-cell} yields $\log\E e^{\alpha_0 Z} \le C \le C r_0^{-\rho}r^\rho \le C' r^\rho$. For $r < r_0$, we partition the expectation over $\mathcal A=\{x_K\le\sqrt r\}$ and its complement. On the near-set, $Z\le \varepsilon_r(I_\lambda(K)+1)$ with $\varepsilon_r = Cwr \le C r^{1/2}$. Applying $e^x - 1 \le x e^x$ yields the variance contraction $\E[(e^{\alpha_0 Z}-1)\1_{\mathcal A}] \le \alpha_0 \varepsilon_r \E[(I_\lambda(K)+1) e^{\alpha_0 \varepsilon_r (I_\lambda(K)+1)}]$. By decreasing $\alpha_0$ such that $\alpha_0 \varepsilon_r < \theta$, Lemma~\ref{lem:poisson-information} bounds the expectation, yielding $\E[(e^{\alpha_0 Z}-1)\1_{\mathcal A}] \le C r^{1/2}$. On the far-set $\mathcal A^c$, the far-estimate guarantees $\alpha_0 Z \le \theta I_\lambda(K)$ and $I_\lambda(K)\ge c(1+\log(A/r))$. The large deviations absorb the spatial weight; since $\theta < \theta'$, it holds
 $\E[(e^{\alpha_0 Z}-1)\1_{\mathcal A^c}] \le \E[e^{\theta I_\lambda(K)}\1_{\{I_\lambda(K)\ge c(1+\log(A/r))\}}] \le e^{-(\theta'-\theta)c(1+\log(A/r))} \E e^{\theta' I_\lambda(K)} \le Cr^{\rho_1}$. Summing these components and bounding $\log(1+x)\le x$ completes the proof with $\rho = \min\{1/2, \rho_1\}$.
\end{proof}

\begin{lemma}
\label{lem:gaussian-cells}
Under the bound~\eqref{eq:dens_Gaussbound}, there exist $A_0,c_1>0$ such that uniformly in $0<h\le1$, the shifts and the cells, it holds that $r_m^U\le A_0e^{-c_1w_m^U}$. Consequently, $w_m^U\le c_1^{-1}\log(A_0/r_m^U)$ for $r_m^U>0$, and the lattice sum satisfies $\sum_m(r_m^U)^\rho\le C_\rho h^{-d}$ for any $\rho>0$.
\end{lemma}
\begin{proof}
Let $R_m=\sup_{x\in B_m^U}|x|$ and $\underline R_m=\inf_{x\in B_m^U}|x|$. Since the cell diameter is at most $\sqrt d$, $\underline R_m^2\ge\tfrac12R_m^2-d$. Thus $r_m^U\le C_0e^{-c_0\underline R_m^2} \le C_0e^{c_0d}e^{-(c_0/2)R_m^2}$, proving the cell envelope and resulting logarithmic bound. Bounding the discrete sum by the continuous integral of a Gaussian over each cell yields $\sum_m(r_m^U)^\rho\le C_\rho h^{-d}$, uniformly in the shift.
\end{proof}

\begin{proof}[Proof of Theorem \ref{thm:weighted-mgf}]
First note that if \(L_\Xi=0\), then \(\Xi\) is constant, so the conclusion is immediate. We therefore assume in the following that \(L_\Xi>0\). 
For a shift \(U_\ell\), denote \(r_{m,s}^{U_\ell}:=v^{-1}\int_{B_m^{U_\ell}} p_s(y)\,dy\); whenever \(x\in B_m^{U_\ell}\), it holds \(r_{m,s}^{U_\ell}=\Pi_{h,U_\ell}p_s(x)\). Denote for the discrepancy on grid \(\ell\) seen by particle \(i\),
\[
a_{i,\ell,s}(x)
=\cC\!\left( \frac{n_{m(i,\ell),\ell}(x)-1}{(N-1)v}\right)
-\cC(r_{m(i,\ell),s}^{U_\ell}),
\quad\text{i.e.,}\quad
\hRop(x_i;x)-\Rop p_s(x_i)
=\frac{1}{L}\sum_{\ell=1}^L a_{i,\ell,s}(x),
\]
and define the full empirical fluctuation energy 
\begin{equation}\label{eq:fluctuation-energy}
\mathcal{S}_s^N(x)
=\sum_{i=1}^N (1+|x_i|^2)\left|
\Xi(\hRop(x_i;x))-\Xi(\Rop p_s(x_i))
\right|^2.
\end{equation}
By the \(L_\Xi\)-Lipschitz continuity of \(\Xi\) and Jensen's inequality, 
\[
\left|
\Xi(\hRop(x_i;x))-\Xi(\Rop p_s(x_i))
\right|^2
\leq L_\Xi^2\left|
\frac{1}{L}\sum_{\ell=1}^L a_{i,\ell,s}(x)
\right|^2
\leq \frac{L_\Xi^2}{L}
\sum_{\ell=1}^L |a_{i,\ell,s}(x)|^2,
\]
yielding 
\(
\mathcal{S}_s^N \leq L_\Xi^2 L^{-1} \sum_{\ell=1}^L \mathcal{S}_{N,s}^{(\ell)}(x)
\)
with
\[
\mathcal{S}_{N,s}^{(\ell)}(x)
=\sum_{m\in\mathbb{Z}^d} n_{m}^{U_\ell}(x) w_m^{U_\ell} \left|
\cC\!\left( \frac{(n_{m}^{U_\ell}(x)-1)_+}{(N-1)v}\right)
-\cC(r_{m,s}^{U_\ell})
\right|^2,
\quad 
w_m^{U_\ell}=1+\sup_{y\in B_m^{U_\ell}}|y|^2,
\]
the one-grid energy corresponding to the shift \(U_\ell\). To prove the theorem, it hence remains to bound the exponential moment of the remaining sum. To this end, note that for a fixed grid \(\ell\), the cell counts have the multinomial distribution 
\((n_m^{U_\ell})_m
\sim
\operatorname{Mult}\bigl(N,(P_{m,\ell}(s))_m\bigr)
\)
with
\(
P_{m,\ell}(s)=vr_{m,\ell}(s)
\). 
Cells where \(r_{m,\ell}(s) = 0\) enforce \(P_{m,\ell} = 0\), leading to \(n_m^{U_\ell} = 0\) almost surely, yielding a trivial \(0=0\) error contribution. In case \(r_{m,\ell}>0\), we Poissonize the system to move from dependent multinomial cell counts to independent variables, i.e., draw for each cell independent \(K_{m,\ell}\sim\operatorname{Poi}(\lambda_{m,\ell})\) with \(\lambda_{m,\ell}=(N-1)P_{m,\ell}(s)\).
Since
	\begin{equation}
 \E_{Q_s^N}\left[e^{\alpha_0 \mathcal{S}_{N,s}^{(\ell)}}\right] 
 = \E\left[ e^{\alpha_0 \mathcal{S}_{N,s}^{(\ell)}(K)} \;\middle|\; \sum_m K_{m,\ell}=N \right]
 = \E\left[ \exp\left\{\alpha_0 \sum_m Z_m(K_{m,\ell})\right\} \;\middle|\; \sum_m K_{m,\ell}=N \right],
\end{equation}
where 
\[
Z_m(K_{m,\ell})
:=K_{m,\ell} w_m^{U_\ell}\left|
\cC\left(
\frac{(K_{m,\ell}-1)_+}{(N-1)v}
\right)
-\cC(r_{m,s}^{U_\ell})
\right|^2,
\]
we recover the original distribution. By Lemma~\ref{lem:gaussian-cells}, it holds that \(w_m^{U_\ell}\leq c_1^{-1}\log(A_0/r_{m,s}^{U_\ell})\) for some \(A_0,c_1>0\) uniformly in shifts, cells and their sizes. On the other hand, the pointwise bound \(p_s\le C_0\) implies
\(0<r_{m,s}^{U_\ell}\le C_0<R\). Moreover,
\[
\log\!\left(\frac{A_0}{r_{m,s}^{U_\ell}}\right)
\le
\left|\log\!\left(\frac{A_0}{C_0}\right)\right|
+
\log\!\left(\frac{C_0}{r_{m,s}^{U_\ell}}\right)
\le
C\left(
1+\log_+\!\left(\frac{C_0}{r_{m,s}^{U_\ell}}\right)
\right).
\]
Thus the assumptions of Lemmas~\ref{lem:deterministic-cell}
and~\ref{lem:one-cell} hold with \(A=C_0<R\) and with an
enlarged constant \(B\), and consequently,
\[
\log\E[\exp(\alpha_0 Z_m(K_{m,\ell}))]\leq C(r_{m,s}^{U_\ell})^\rho
\quad \text{with some}\quad \alpha_0, C, \rho>0.
\] 
Choosing increasing finite sets \(F_j\to\mathbb{Z}^d\), since \(Z_m\geq 0\) and 
\(\sum_{m\in F_j} Z_m(K_{m,\ell})\to \mathcal{S}_{N,s}^{(\ell)}(K)\), it then follows by monotone convergence that
\[
\E[e^{\alpha_0 \mathcal{S}_{N,s}^{(\ell)}(K)}]
=\lim_{j\to\infty} \E\left[
\exp\left(
\alpha_0 \sum_{m\in F_j} Z_m(K_{m,\ell})
\right)
\right]
\leq \exp\left(
C\sum_m (r_{m,s}^{U_\ell})^\rho
\right)
\leq \exp(C_\rho h^{-d}),
\]
the last inequality again following from Lemma~\ref{lem:gaussian-cells}. From this, using the Stirling's formula to bound \(\mathbb{P}(\sum_m K_{m,\ell}=N)\), we get
\[
\log\E_{Q_s^N}[e^{\alpha_0 \mathcal{S}_{N,s}^{(\ell)}}]
=\log \frac{\E[e^{\alpha_0 \mathcal{S}_{N,s}^{(\ell)}(K)}\1_{\{\sum_m K_{m,\ell}=N\}}]}{\mathbb{P}(\sum_m K_{m,\ell}=N)}
\leq C_\rho h^{-d} + \frac{1}{2}\log N + C
\]
for some constant \(C>0\). Taking \(\alpha_*=\alpha_0/L_\Xi^2\), the H\"older inequality yields
\[
\log \E_{Q_s^N} [e^{\alpha_* \mathcal{S}_s^N}]
\leq \log \E_{Q_s^N} \left[ 
\prod_{\ell=1}^L e^{\alpha_0 \mathcal{S}_{N,s}^{(\ell)}/L}
\right]
\leq \log \prod_{\ell=1}^L \left(
\E_{Q_s^N} e^{\alpha_0 \mathcal{S}_{N,s}^{(\ell)}}
\right)^{1/L}
\leq C_*(h^{-d}+\log N)
\]
with some constant \(C_*>0\), concluding the proof.
\end{proof}

\section{Proofs of Main Text Results}
\label{app:proofs}

\begin{proof}[Proof of Lemma~\ref{lem:weighted-bias}]
Because $p_t\le C_0<R$, clipping is inactive at $p_t(x)$ and at every local cell average. First, the Gaussian density envelope alone provides a zero-order estimate. Because the cells are half-open cubes of diameter at most $\sqrt{d}\,h$, any $y \in B_m^{U_\ell}$ containing $x$ satisfies $|y| \ge (|x| - \sqrt{d}\,h)_+$. Integrating over the cell gives
\begin{equation}
    \Pi_{h,U_\ell}p_t(x) \le C_0 \exp\!\left[-c_0(|x|-\sqrt{d}\,h)_+^2\right] \le C_2 e^{-c'|x|^2}, \qquad 0<h\le 1,
\end{equation}
with some \(C_2>0\)
uniformly in the shift. This demonstrates that cell averages inherit Gaussian decay; the constants in this cell average bound may be larger than the original $C_0$, but the qualitative decay is preserved. Using the Lipschitz continuity of $\Xi$, we have 
\begin{equation}\label{eq:Lipschitz-Xi}
	|\Xi(\Rop p_t(x))-\Xi(p_t(x))| \le L_\Xi|\Rop p_t(x)-p_t(x)| \le C_3 e^{-c'|x|^2}.
\end{equation}
Multiplying this squared difference by the linear growth bound $|\nabla\Phi(x)|^2\le C_\Phi(1+|x|^2)$, the Gaussian decay absorbs the polynomial growth $(1+|x|^2)e^{-2c'|x|^2} \le 1+1/(2c'e)$, yielding the global bound 
\begin{equation}\label{eq:beta-bound1}
	\beta_h(t) \le C_4,
	\qquad C_4:=C_\Phi C_3^2(1+1/(2c'e)),
\end{equation}
uniformly on \(0<t\leq T\). On the other hand,
for the gradient-based estimate, the mean value theorem gives for any $x,y$ within the same $h$-cell: $|p_t(y)-p_t(x)| \le |y-x|\int_0^1 |\nabla p_t(x+\theta(y-x))|\dd\theta$. Since the intra-cell distance is bounded by $|x-y|\le\sqrt d\,h$, the gradient bound from Assumption~\ref{ass:density} implies $|\Pi_{h,U}p_t(x)-p_t(x)| \le C_5 h \max\{1,t^{-1/2}\} (1+|x|)^m e^{-a_2|x|^2}$ with \(a_2<a_1\) and \(C_5>0\) depending on \(C_1\), \(d\), \(m\), \(a_1\) and \(a_2\), uniformly in $h$ and $U$. Applying Jensen's inequality, we get, similar to~\eqref{eq:Lipschitz-Xi}, $|\Xi(\Rop p_t(x))-\Xi(p_t(x))|^2 \le C_6 h^2 \max\{1,t^{-1}\}(1+|x|)^{2m}e^{-2a_2|x|^2}$, and, since the polynomial growth of \(|\nabla\Phi(x)|^2\) is again absorbed by the exponential tail for some \(a'<a_2\), obtain
\begin{equation}\label{eq:beta-bound2}
	\beta_h(t) \le C h^2\max\{1,t^{-1}\}.
\end{equation}
for some generic \(C>0\).
Observing \(\max\{1,t^{-1}\}\le \max\{1,T\}/t\) and combining~\eqref{eq:beta-bound1} and~\eqref{eq:beta-bound2}, we conclude the claim.
\end{proof}

\begin{proof}[Proof of Theorem~\ref{thm:main-poc} and Corollaries]
%We track $H_N(t)=\Ent(P_t^N\mid Q_t^N)$ against the independent product measure $Q_t^N=p_t^{\otimes N}$. 
%Under the reference law, the uncoupled drift for particle $i$ is $\overline b_t(x_i)=-\Xi(p_t(x_i))\nabla\Phi(x_i)$. 
Using Lemma~\ref{lem:entropy-production} and
inserting the population histogram field $\Rop p_t$ between the empirical estimator and the true density, the Lipschitz continuity of $\Xi$ and Lemma~\ref{lem:weighted-bias} decompose the drift error as
\begin{equation}\label{eq:drift-error-decomposition}
|b_i^N(x)-\overline b_t(x_i)|^2 \le C(1+|x_i|^2) |\Xi(\hRop(x_i;x))-\Xi(\Rop p_t(x_i))|^2 + C\beta_h(t). 
\end{equation}
Summing the first term over all particles yields $\mathcal S_t^N(x)$ defined in~\eqref{eq:fluctuation-energy}, from which
\begin{equation}
 H_N(t) \le C \int_0^t \E_{P_s^N}[\mathcal S_s^N] \dd s + C N \int_0^t \beta_h(s) \dd s,
\end{equation}
where \(H_N(t)=\Ent(P_t^N\mid Q_t^N)\).
Applying the Donsker-Varadhan variational principle inside the integral gives 
\begin{equation}\label{eq:S-bound}
\E_{P_s^N}[\mathcal S_s^N] \le \frac{1}{\alpha_*}H_N(s) + \frac{1}{\alpha_*}\log\E_{Q_s^N} \bigl[e^{\alpha_*\mathcal S_s^N}\bigr]
\le \frac{1}{\alpha_*}H_N(s) + \frac{C_*}{\alpha_*} (h^{-d}+\log N),
\end{equation}
with \(\alpha_*, C_*>0\) fixed uniform on \([0,T]\) constants given in Theorem~\ref{thm:weighted-mgf}, from which
\begin{equation}
 H_N(t) \le C \int_0^t H_N(s) \dd s + C t(h^{-d} + \log N) + C N \int_0^t \beta_h(s) \dd s
\end{equation}
with some generic constant \(C>0\). Given Lemma~\ref{lem:weighted-bias}, enlarging the constant \(C_T\) if necessary, we get for \(\beta_h(s)\)
\begin{equation}
\label{eq:int-beta-bound}
 \int_0^t \beta_h(s) \dd s \le C_T \int_0^t \min\left\{1,\frac{h^2}{s}\right\} \dd s \le C_T h^2\left(1+\log_+\frac{t}{h^2}\right) \le C_T h^2 \bigl(1+|\log h|\bigr)
\end{equation}
The Gronwall's inequality then yields~\eqref{eq:full-entropy-main} with a constant incorporating the constant above as well as exponential growth in \(T\) resulting from application of Gronwall's inequality. The inequalities~\eqref{eq:local-entropy-main} and~\eqref{eq:tv-main} then follow by the argument similar to~\cite{Lacker2023}. More precisely, the entropy version of Shearer's inequality applied to all $k$-coordinate marginals, combined with the symmetry of $P_t^N$, gives $k^{-1}\Ent(P_t^{N,k}\mid p_t^{\otimes k}) \le N^{-1}H_N(t)$, yielding \eqref{eq:local-entropy-main}, and applying Pinsker's inequality translates this into the total variation metric~\eqref{eq:tv-main}. Observing that
\[
\mathbb{E}_{P_t^N}[
(1+|X_t^{i,N}|^2) |\Xi(\hRop(X_t^{i,N};X_t))-\Xi(\Rop p_t(X_t^{i,N}))|^2
]
=\frac{1}{N}\mathbb{E}_{P_t^N}[\mathcal S_t^N(X_t)]
\]
by exchangeability, the remaining bound~\eqref{eq:drift-chaos-main} then follows from~\eqref{eq:drift-error-decomposition}, \eqref{eq:S-bound}, \eqref{eq:full-entropy-main} and~\eqref{eq:int-beta-bound}.

Setting the optimally balanced bandwidth $h_N\asymp(N\log N)^{-1/(d+2)}$ leads to
\begin{equation}
h^2 |\log h| \asymp N^{-2/(d+2)}(\log N)^{d/(d+2)} \quad \text{and} \quad N^{-1}h^{-d} \asymp N^{-2/(d+2)}(\log N)^{d/(d+2)}.
\end{equation}
Since $d\ge1$, the residual term $(\log N)/N$ is bounded by a constant multiple of this rate, balancing the terms and yielding the rates in Corollary~\ref{cor:optimized}.
Provided $k_N N^{-2/(d+2)}(\log N)^{d/(d+2)} \to 0$ with $1 \le k_N \le N$, the relative entropy vanishes asymptotically, extending the propagation of chaos to growing ensembles in Corollary~\ref{cor:growing-marginals}.
\end{proof}

\end{document}